\documentclass[12pt]{article}
\usepackage{amsmath, amsthm, amssymb}
\usepackage{hyperref}
\usepackage{verbatim}
\usepackage[top=1.0in, bottom=1.0in, left=1.0in, right=1.0in]{geometry}

\usepackage{tkz-graph}
\usetikzlibrary{arrows}
\usetikzlibrary{shapes}
\usepackage[position=bottom]{subfig}

\usepackage{longtable}
\usepackage{array}

\usepackage{sectsty}
\allsectionsfont{\sffamily}

\makeatletter
\newtheorem*{rep@theorem}{\rep@title}
\newcommand{\newreptheorem}[2]{
\newenvironment{rep#1}[1]{
 \def\rep@title{#2 \ref{##1}}
 \begin{rep@theorem}}
 {\end{rep@theorem}}}
\makeatother

\theoremstyle{plain}
\newtheorem{thm}{Theorem}[section]
\newreptheorem{thm}{Theorem}

\newreptheorem{prop}{Proposition}
\newtheorem{lem}[thm]{Lemma}
\newreptheorem{lem}{Lemma}
\newtheorem{conjecture}[thm]{Conjecture}
\newreptheorem{conjecture}{Conjecture}

\newreptheorem{cor}{Corollary}

\newtheorem*{SmallPotLemma}{Small Pot Lemma}

\theoremstyle{definition}
\newtheorem{defn}{Definition}
\theoremstyle{remark}

\newtheorem*{problem}{Problem}

\newcommand{\fancy}[1]{\mathcal{#1}}

\newcommand{\IN}{\mathbb{N}}

\newcommand{\CC}{\fancy{C}}
\newcommand{\D}{\fancy{D}}

\newcommand{\surj}{\twoheadrightarrow}

\newcommand{\set}[1]{\left\{ #1 \right\}}
\newcommand{\setb}[3]{\left\{ #1 \in #2 \mid #3 \right\}}
\newcommand{\setbs}[2]{\left\{ #1 \mid #2 \right\}}
\newcommand{\card}[1]{\left|#1\right|}
\newcommand{\size}[1]{\left\Vert#1\right\Vert}
\newcommand{\ceil}[1]{\left\lceil#1\right\rceil}

\newcommand{\func}[3]{#1\colon #2 \rightarrow #3}

\newcommand{\funcsurj}[3]{#1\colon #2 \surj #3}
\newcommand{\irange}[1]{\left[#1\right]}
\newcommand{\join}[2]{#1 \mbox{\hspace{2 pt}$\ast$\hspace{2 pt}} #2}

\newcommand{\parens}[1]{\left( #1 \right)}
\newcommand{\brackets}[1]{\left[ #1 \right]}
\newcommand{\DefinedAs}{\mathrel{\mathop:}=}

\def\D{\fancy{D}}

\title{Coloring graphs with dense neighborhoods}
\author{Landon Rabern}
\date{\today}

\begin{document}
\maketitle

\begin{abstract}
It is shown that any graph with maximum degree $\Delta$ in which the average degree of the induced subgraph on the set of all neighbors of any vertex exceeds $\frac{6k^2}{6k^2 + 1}\Delta + k + 6$ is either $(\Delta - k)$-colorable or contains a clique on more than $\Delta - 2k$ vertices.  In the $k=1$ case we improve the bound on the average degree to $\frac23\Delta + 4$ and the bound on the clique number to $\Delta-1$.  As corollaries, we show that every graph satisfies $\chi \leq \max\set{\omega, \Delta - 1, 4\alpha}$ and every graph satisfies $\chi \leq \max\set{\omega, \Delta - 1, \ceil{\frac{15 + \sqrt{48n + 73}}{4}}}$.
\end{abstract}

\section{Introduction}
Using ideas developed for strong coloring by Haxell \cite{haxell2004strong} and by Aharoni, Berger and Ziv \cite{aharoni2007independent}, we make explicit a recoloring technique and apply it to coloring graphs with dense neighborhoods.  The \emph{average degree} of a graph $G$ is $d(G) \DefinedAs \frac{2\size{G}}{\card{G}}$. For a vertex $v$ in a graph $G$, put $G_v \DefinedAs G\brackets{N(v)}$.  Reed \cite{reed1998omega} has conjectured that every graph satisfies

\[\chi \leq \ceil{\frac{\omega + \Delta + 1}{2}}.\]

Our first result implies this conjecture without the round-up for graphs where every vertex is in a big clique.

\begin{repthm}{MainSimpleCorollary}
Let $k \geq 1$. Every graph $G$ with $\omega(G) \leq \Delta(G) - 2k$ such that every vertex is in a clique on $\frac{2k}{2k+1}\Delta(G) + 2k + 1$ vertices is $(\Delta(G)-k)$-colorable.
\end{repthm}

Using probabilistic methods, Reed \cite{reed1998omega} proved a similar-looking result that is much better for very large $k$ and $\Delta$.  For comparison, we modify the statement to look as close to Theorem \ref{MainSimpleCorollary} as possible.

\begin{thm}[Reed \cite{reed1998omega}]\label{ReedOmega}
There exists $\Delta_0$ such that for $k \geq 0$ every graph $G$ with $\Delta(G) \geq \Delta_0$ and 
$\frac{69999999}{70000000}\Delta(G) \leq \omega(G) \leq \Delta(G) - 2k$ is $(\Delta(G)-k)$-colorable.
\end{thm}

This implies Theorem \ref{MainSimpleCorollary} when $k$ gets to be larger than around $35$ million.  In fact, Reed states that with some care the constant can be brought down to $\frac{9999}{10000}$ and so really his method starts implying Theorem \ref{MainSimpleCorollary} when $k$ gets larger than $5000$.  Moreover, Theorem \ref{ReedOmega} just needs a large clique while Theorem \ref{MainSimpleCorollary} requires every vertex to be in a large clique.
 
It turns out that if every neighborhood has many edges, it is guaranteed that every vertex is in a large clique.  This implies the following.

\begin{repthm}{MainResult}
Let $k \geq 0$. Every graph $G$ with $\omega(G) \leq \Delta(G) - 2k$ such that $d(G_v) \geq \frac{6k^2}{6k^2 + 1}\Delta(G) + k + 6$ for each $v \in V(G)$ is $(\Delta(G)-k)$-colorable.
\end{repthm}

To tighten these results up, further development of the theory of $f$-choosability where $f(v) = d(v) - k$ for $k \geq 2$ is needed.  For $k=1$ this theory was developed in \cite{mules} and using it in the case of $(\Delta-1)$ coloring, we achieve tighter bounds which have bearing on the conjecture of Borodin and Kostochka \cite{borodin1977upper}.

\begin{conjecture}[Borodin and Kostochka \cite{borodin1977upper}]\label{BK}
Every graph with $\chi \geq \Delta \geq 9$ contains $K_\Delta$.
\end{conjecture}

Also using probabilistic methods, Reed \cite{reed1999strengthening} has proved this conjecture for very large $\Delta$.  Using $d_1$-choosability theory, we prove the following.

\begin{repthm}{TwoThirdsCliqueCor}
Every graph with $\chi \geq \Delta \geq 9$ such that every
vertex is in a clique on $\frac23\Delta + 2$ vertices contains $K_\Delta$.
\end{repthm}

From this it follows that it would be enough to prove the Borodin-Kostochka conjecture for irregular graphs.

\begin{repthm}{IrregularReduction}
Every graph satisfying $\chi \geq \Delta = k \geq 9$ either
contains $K_k$ or contains an irregular critical subgraph satisfying $\chi
= \Delta = k - 1$.
\end{repthm}

We also get a neighborhood density version.

\begin{repthm}{BKdense}
Every graph $G$ with $\omega(G) < \Delta(G)$ such that $d(G_v) \geq \frac23\Delta(G) + 4$ for each $v \in V(G)$ is $(\Delta(G)-1)$-colorable.
\end{repthm}

Finally, we use these ideas to prove the following bounds on the chromatic number.  The first generalizes the result of Beutelspacher and Hering \cite{beutelspacher1984minimal} that the Borodin-Kostochka conjecture holds for graphs with independence number at most two.  This result was generalized in another direction in \cite{cranstonrabernclaw} where the conjecture was proved for claw-free graphs.

\begin{repthm}{AlphaBound}
Every graph satisfies $\chi \leq \max\set{\omega, \Delta - 1, 4\alpha}$.
\end{repthm}

The second bound shows that the Borodin-Kostochka conjecture holds for graphs with maximum degree on the order of the square root of their order.  This improves on prior bounds of $\Delta > \frac{n + 1}{2}$ from Beutelspacher and Hering \cite{beutelspacher1984minimal} and $\Delta > \frac{n-6}{3}$ of Naserasr \cite{naserasr}.

\begin{repthm}{OrderBound}
Every graph satisfies $\chi \leq \max\set{\omega, \Delta - 1, \ceil{\frac{15 + \sqrt{48n + 73}}{4}}}$.
\end{repthm}

\section{Strong coloring}
For a positive integer $r$, a graph $G$ with $\card{G} = rk$ is called \emph{strongly $r$-colorable} if for every partition of $V(G)$ into parts of size $r$ there is a proper coloring of $G$ that uses all $r$ colors on each part.  If $\card{G}$ is not a multiple of $r$, then $G$ is strongly $r$-colorable iff the graph formed by adding $r\ceil{\frac{|G|}{r}} - |G|$ isolated vertices to $G$ is strongly $r$-colorable.  The \emph{strong chromatic number} $s\chi(G)$ is the smallest $r$ for which $G$ is strongly $r$-colorable.

Note that a strong $r$-coloring of $G$ with respect to a partition $V_1, \ldots, V_k$ of $V(G)$ with $\card{V_i} = r$ must partition $V(G)$ into $r$ independent transversals of $V_1, \ldots, V_k$. In \cite{szabo2006extremal}, Szab{\'o} and Tardos constructed partitioned graphs with part sizes $2\Delta - 1$ that have no independent transversal.  So we must have $s\chi(G) \geq 2\Delta(G)$.  It is conjectured that this bound is tight.

Haxell \cite{haxell2004strong} proved that $s\chi(G) \leq 3\Delta(G) - 1$.  Aharoni, Berger and Ziv \cite{aharoni2007independent} gave a simple proof that $s\chi(G) \leq 3\Delta(G)$.  It is this latter proof whose recoloring technique we use.  First we need a lemma allowing us to pick an independent transversal when one of the sets has only one element.

\begin{lem}\label{SingletonSetTransversal}
Let $H$ be a graph and $V_1 \cup \cdots \cup V_r$ a partition of $V(H)$. 
Suppose that $\card{V_i} \geq 2\Delta(H)$ for each $i \in \irange{r}$.  If a
graph $G$ is formed by attaching a new vertex $x$ to fewer than $2\Delta(H)$
vertices of $H$,  then $G$ has an independent set $\set{x, v_1, \ldots, v_r}$
where $v_i \in V_i$ for each $i \in \irange{r}$.
\end{lem}
\begin{proof}
Suppose not. Remove $\set{x} \cup N(x)$ from $G$ to form $H'$ with induced
partition $V_1', V_2', \ldots, V_r'$. Then $V_1', V_2', \ldots, V_r'$ has no
independent transversal since we could combine one with $x$ to get our desired
independent set in $G$. Note that $\card{V_i'} \geq 1$. 
Create a graph $Q$ by removing edges from $H'$ until it is edge minimal without
an independent transversal. Pick $yz \in E(Q)$ and apply Lemma
\ref{BaseTransversalLemma} on $yz$ with the induced partition to get the guaranteed 
$J \subseteq \irange{r}$ and the totally dominating induced matching 
$M$ with $\card{M} = \card{J} - 1$. 
Now $\card{\bigcup_{i \in J} V_i'} > 2\Delta(H)\card{J} - 2\Delta(H) =
2(\card{J} - 1)\Delta(H)$ and hence $M$ cannot dominate, a contradiction.
\end{proof}

\begin{thm}\label{StrongColorBound}
Every graph satisfies $s\chi \leq 3\Delta$.
\end{thm}
\begin{proof}
We only need to prove that graphs with $n \DefinedAs 3\Delta k$ vertices have a $3\Delta$ coloring for each $k \geq 1$.  
Suppose not and choose a counterexample $G$ minimizing $\size{G}$.  Put $r \DefinedAs 3\Delta(G)$ and let $V_1, \ldots, V_k$ be a partition 
of $G$ for which there is no acceptable coloring.  Then the $V_i$ are independent by minimality of $\size{G}$. By symmetry we may assume 
there are adjacent vertices $x \in V_1$ and $y \in V_2$. Apply minimality of $\size{G}$ to get an $r$-coloring $\pi$ of $G - xy$ 
with $\pi(V_i) = \irange{r}$ for each $i \in \irange{k}$.  We will modify $\pi$ to get such a coloring of $G$.

By symmetry, we may assume that $\pi(x) = \pi(y) = 1$.  For $2 \leq i \leq k$, let $z_i$ be the unique element of $\pi^{-1}(1) \cap V_i$ and 
put $W_i \DefinedAs V_i - \setb{v}{V_i}{\pi(v) = \pi(w) \text{ for some } w \in N(z_i)}$.  Then $\card{W_i} \geq 2\Delta(G)$ and we may apply 
Lemma \ref{SingletonSetTransversal} to get a $G$-independent transversal $w_1, w_2, \ldots, w_k$ of $\set{x}, W_2, W_3, \ldots, W_k$.  
Define a new coloring $\zeta$ of $G$ by

\begin{equation*}
\zeta(v) \DefinedAs 
\begin{cases}
1 & \text{if $v = w_i$} \\
\pi(w_i) & \text{if $v = z_i$} \\
\pi(v) & \text{otherwise.}
\end{cases}
\end{equation*}

\noindent Then $\zeta$ is a proper coloring of $G$ with $\zeta(V_i) = \irange{r}$ for each $i \in \irange{k}$, a contradiction.
\end{proof}

For our application we will need a lopsided version of Lemma \ref{SingletonSetTransversal} generalizing King's \cite{KingHitting} lopsided version of Haxell's lemma.

\begin{lem}\label{SingletonSetTransversalLopsided}
Let $H$ be a graph and $V_1 \cup \cdots \cup V_r$ a partition of $V(H)$.  
Suppose there exists $t \geq 1$ such that for each $i \in \irange{r}$ and each $v \in V_i$ we have $d(v) \leq \min\set{t, \card{V_i}-t}$.  For any $S \subseteq V(H)$ with $\card{S} < \min\set{\card{V_1}, \ldots, \card{V_r}}$, there is an independent transversal $I$ of $V_1, \ldots, V_r$ with $I \cap S = \emptyset$.
\end{lem}
\begin{proof}
Suppose the lemma fails for such an $S \subseteq V(H)$.  Put $H' \DefinedAs H - S$ and let $V_1', \ldots, V_r'$ be the induced partition of $H'$. Then there is no independent trasversal of $V_1', \ldots, V_r'$ and $\card{V_i'} \geq 1$ for each $i \in \irange{r}$. Create a graph $Q$ by removing edges from $H'$ until it is edge minimal without an independent transversal. Pick $yz \in E(Q)$ and apply Lemma 
\ref{BaseTransversalLemma} on $yz$ with the induced partition to get the guaranteed 
$J \subseteq \irange{r}$ and the tree $T$ with vertex set $J$ and an edge between $a, b \in
J$ for each $uv \in M$ with $u \in V_a'$ and $v \in V_b'$.  By our condition, for each $uv \in E(V_i, V_j)$, we have $\card{N_H(u) \cup N_H(v)} \leq \min\set{\card{V_i}, \card{V_j}}$.

Choose a root $c$ of $T$. Traversing $T$ in leaf-first order and for each leaf $a$ with parent $b$ picking $|V_a|$ from $\min\set{|V_a|, |V_b|}$ we get that the vertices in $M$ together dominate at most $\sum_{i \in J - c} \card{V_i}$ vertices in $H$.  Since $\card{S} < \card{V_c}$, $M$ cannot totally dominate $\bigcup_{i \in J} V_i'$, a contradiction.
\end{proof}

We note that the condition on $S$ can be weakened slightly.  Suppose we have ordered the $V_i$ so that $\card{V_1} \leq \card{V_2} \leq \cdots \leq \card{V_r}$.  Then for any $S \subseteq V(H)$ with $\card{S} < \card{V_2}$ such that $V_1 \not \subseteq S$, there is an independent transversal $I$ of $V_1, \ldots, V_r$ with $I \cap S = \emptyset$.  The proof is the same except when we choose our root $c$, choose it so as to maximize $\card{V_c}$.  Since $\card{J} \geq 2$, we get $\card{V_c} \geq \card{V_2} > \card{S}$ at the end.

\section{The recoloring technique}\label{recolorsection}
We can extract the idea in the proof of Theorem \ref{StrongColorBound} to get a
general recoloring technique.  Suppose $G$ is a $k$-vertex-critical graph and
pick $x \in V(G)$ and $(k-1)$-coloring $\pi$ of $H \DefinedAs G - x$.  Let $Z$
be a color class of $\pi$, say $Z = \pi^{-1}(1)$.  For each $z \in Z$, let
$O_z$ be the neighbors of $z$ which get a color that no other neighbor of $z$ gets; that is, put
$O_z \DefinedAs \setb{v}{N_H(z)}{\pi(v) \not \in \pi(N_H(z) - v)}$.  Suppose the
$O_z$ are pairwise disjoint.  If we could find an independent transversal
$\set{x} \cup \set{v_z}_{z \in Z}$ of $\set{x}$ together with the $O_z$, then
recoloring each $z \in Z$ with $\pi(v_z)$ and coloring each vertex in $\set{x}
\cup \set{v_z}_{z \in Z}$ with $1$ gives a proper $(k-1)$-coloring of $G$.  This
is exactly what happens in the above proof of the strong coloring result.  To
make this work more generally, we need to find situations where each $G[O_z]$
has high minimum degree.  Also, intuitively, the $O_z$ intersecting each other
should make things easier since recoloring a vertex in the intersection of
$O_{z_1}$ and $O_{z_2}$ works for both $z_1$ and $z_2$.  In our applications we
will allow some restricted intersections.

\section{Borodin-Kostochka when every vertex is in a big clique}
The case of $(\Delta-1)$-coloring is easier and provides a good warm-up for general coloring.  Also, we achieve tighter bounds in this case because the list coloring theory is more developed.

\subsection{A general decomposition}\label{GeneralDecomposition}
Let $\D_1$ be the collection of graphs without induced $d_1$-choosable
subgraphs.  Plainly, $\D_1$ is hereditary. For a graph $G$ and $t \in \IN$, let
$\CC_t$ be the maximal cliques in $G$ having at least $t$ vertices. We prove the
following decomposition result for graphs in $\D_1$ which generalizes Reed's decomposition in \cite{reed1999strengthening}.  

\begin{lem}\label{partition}
Suppose $G \in \D_1$ has $\Delta(G) \geq 8$ and contains no $K_{\Delta(G)}$. If
$\frac{\Delta(G) + 5}{2} \leq t \leq \Delta(G) - 1$, then $\bigcup \CC_t$ can be
partitioned into sets $D_1, \ldots, D_r$ such that for each $i \in \irange{r}$
at least one of the following holds:
\begin{itemize}
  \item $D_i = C_i \in \CC_t$,
  \item $D_i = C_i \cup \set{x_i}$ where $C_i \in \CC_t$ and $\card{N(x_i) \cap
  C_i} \geq t-1$.
\end{itemize}

\noindent Moreover, each $v \in V(G) - D_i$ has at most $t-2$ neighbors in $C_i$  for each $i \in \irange{r}$.
\end{lem}
\begin{proof}
Suppose $\card{C_i} \leq \card{C_j}$ and $C_i \cap C_j \neq \emptyset$. 
Then $\card{C_i \cap C_j} \geq \card{C_i} + \card{C_j} - (\Delta + 1) \geq 4$.  It follows from Corollary
\ref{K_tClassification} that $\card{C_i - C_j} \leq 1$.

Now suppose $C_i$ intersects $C_j$ and $C_k$.  By the above,
$\card{C_i \cap C_j} \geq \frac{\Delta(G) + 3}{2}$ and similarly $\card{C_i \cap
C_k} \geq \frac{\Delta(G) + 3}{2}$.  Hence $\card{C_i \cap C_j \cap C_k} \geq
\Delta(G) + 3 - (\Delta(G) - 1) = 4$.  Put $I \DefinedAs C_i \cap C_j \cap C_k$
and $U \DefinedAs C_i \cup C_j \cup C_k$.  By maximality of $C_i, C_j, C_k$,
$U$ cannot induce an almost complete graph.  Thus, by Corollary
\ref{K_tClassification}, $\card{U} \in \set{4, 5}$ and the graph induced on $U -
I$ is $E_3$.  But then $t \leq 6$ and hence $\Delta(G) \leq 7$, a contradiction.

\smallskip

\noindent The existence of the required partition is immediate. 
\end{proof}

\noindent When $D_i \in \CC_t$, we put $K_i \DefinedAs C_i \DefinedAs D_i$ and
when $D_i = C_i \cup \set{x_i}$, we put $K_i \DefinedAs N(x_i) \cap C_i$.

\subsection{Doing the recoloring}

Let $G$ be a graph.  For $v \in V(G)$, we let $\omega(v)$ be the size of a
largest clique in $G$ containing $v$.  The proofs of the results in this section
go more smoothly when we strengthen the induction in terms of the parameter
$\rho(G) \DefinedAs \max_{v \in V(G)} d(v) - \omega(v)$.

\begin{lem}\label{MainBKLemma}
For $k \geq 9$, every graph satisfying $\Delta \leq k$, $\omega < k$ and $\rho
\leq \frac{k}{3} - 2$ is $(k-1)$-colorable.
\end{lem}
\begin{proof}
Suppose the theorem fails for some $k \geq 9$ and choose a counterexample
$G$ minimizing $\card{G} + \size{G}$. Put $\Delta \DefinedAs \Delta(G)$.  If
$\Delta < k$, then $\Delta = k-1$ and by Brooks' theorem $G$ contains $K_k$, a
contradiction. Thus $\chi(G) = k = \Delta$.  Also, for any $v
\in V(G)$ we have $\rho(G-v) \leq \rho(G)$, applying our minimality condition on $G$ implies that $G$ is
vertex critical.

Therefore $\delta(G) \geq \Delta - 1$ and $G \in
\D_1$.  For any $v \in V(G)$, we have $\Delta - 1 - \omega(v) \leq d(v) -
\omega(v) \leq \frac{\Delta}{3} - 2$ and hence $\omega(v) \geq \frac23\Delta +
1$. Applying Lemma \ref{partition} with $t \DefinedAs \frac23\Delta + 1$ we
get a partition $D_1, \ldots, D_r$ of $\bigcup \CC_t = V(G)$.  Note that for $i
\in \irange{r}$, if $K_i \neq D_i$ then all vertices in $K_i$ are high by Lemma
\ref{E2JoinWithSomeLow}.  Pick $x \in K_1$.  Then $x$ has $\card{C_1} - 1 \leq
\Delta - 2$ neighbors in $D_1$ if $K_i = D_i$ and $\card{C_1} \leq \Delta - 1$
if $K_i \neq D_i$.  Hence, by our note, $x$ has a neighbor $w \in V(G) - D_1$.

We now claim that $xw$ is a critical edge in $G$.  Suppose otherwise that
$\chi(G - xw) = \Delta$.  Then by minimality of $G$ we must have $\rho(G-xw) >
\rho(G)$. Hence there is some vertex $v \in N(x) \cap N(w)$ so that every
largest clique containing $v$ contains $xw$.  But $v$ is in some $D_j$ and all largest cliques containing $v$ are contained in $D_j$ and hence do not contain $xw$, a contradiction.  

Let $\pi$ be a $(\Delta-1)$-coloring of $G - xw$ chosen so that $\pi(x) = 1$
and so as to minimize $\card{\pi^{-1}(1)}$. Consider $\pi$ as a coloring of
$G-x$. One key property of $\pi$ we will use is that since $x$ got $1$ in the
coloring of $G - xw$ and $x \in K_1$, no vertex of $D_1 - x$ gets colored $1$ by $\pi$.

Now put $Z \DefinedAs \pi^{-1}(1)$ and for $z \in Z$, let $O_z$ be as defined
in Section \ref{recolorsection}.  By minimality of $\card{Z}$, each $z \in Z$ has at least one neighbor in every color class of $\pi$.  
Hence $z$ has two or more neighbors in at most $2 + d(z) - \Delta$ of
$\pi$'s color classes. For each $z \in Z$ we have $i(z)$ such that $z \in D_{i(z)}$. For $z \in Z$ such that $i(z) \not \in i(Z - z)$, put $V_z \DefinedAs O_z \cap C_{i(z)}$.  We have $\card{V_z} \geq \omega(z) - 1 - \parens{2 + d(z) - \Delta}$.  Since $\omega(z) \geq d(z) - \frac13 \Delta + 2$, we have $\card{V_z} \geq \frac23 \Delta - 1$.   Each $y \in V_z$ is adjacent to all of $C_{i(z)} - \set{y}$ and hence has at most $d(y) + 1 - \card{C_{i(z)}}$ neighbors outside $D_{i(z)}$.  Since $\omega(y) \geq d(y) +
2 - \frac13 \Delta$, we conclude that $y$ has at most $d(y) + 1
- (d(y) + 2 - \frac13 \Delta) = \frac13 \Delta - 1$ neighbors outside $D_{i(z)}$.

Now let $Z'$ be the $z \in Z$ with $i(z) \in i(Z - z)$. Then $Z'$ can be partitioned into pairs $\set{z, z'}$ such that $i(z) = i(z')$.  For such a pair, one of $z,z'$ is $x_{i(z)}$ and the other is in $C_{i(z)} - K_{i(z)}$. Put $V_z \DefinedAs O_z \cap O_{z'} \cap K_{i(z)}$ and don't define $V_{z'}$.  We have $\card{V_z} \geq \min\set{\omega(z), \omega(z')} - 1 - \parens{2 + d(z) - \Delta} - \parens{2 + d(z') - \Delta} \geq - \frac13 \Delta + 2 - 1 - 2\parens{2 - \Delta} - \max\set{d(z), d(z')} = \frac53 \Delta - \max\set{d(z), d(z')} - 3 \geq \frac23 \Delta - 3$.  Each $y \in V_z$ is adjacent to all of $D_{i(z)} - \set{y}$ and hence has at most $d(y) + 1 - \card{D_{i(z)}}$ neighbors outside $D_{i(z)}$.  Since $\card{D_{i(z)}} = \omega(y) + 1 \geq d(y) + 3 - \frac13 \Delta$, we conclude that $y$ has at most $\frac13 \Delta - 2$ neighbors outside $D_{i(z)}$.

Let $H$ be the subgraph of $G$ induced on the union of the $V_z$.  Put $S \DefinedAs N(x) \cap V(H)$.  Since $Z \cap D_1 = \emptyset$, $x$ has at least $\card{D_1} - 1$ neighbors in $D_1$ none of which are in $S$.  Hence $\card{S} \leq d(x) + 1 - \card{D_1} \leq d(x) + 1 - \omega(x) \leq \frac{\Delta}{3} - 1 < \card{V_z}$ for all $V_z$ since $\Delta \geq 7$. Hence we may apply Lemma
\ref{SingletonSetTransversalLopsided} on $H$ with $t \DefinedAs \frac13\Delta - 1$ to get an independent set $\set{v_z}_{z\in Z}$ disjoint from $S$ where $v_z \in V_z$. Recoloring each $z \in Z$ with $\pi(z)$ and
coloring $x \cup \set{v_z}_{z \in Z}$ with $1$ gives a $(\Delta - 1)$-coloring
of $G$, a contradiction.
\end{proof}

The following special case is a bit easier to digest.

\begin{thm}\label{TwoThirdsCliqueCor}
Every graph with $\chi \geq \Delta \geq 9$ such that every
vertex is in a clique on $\frac23\Delta + 2$ vertices contains $K_\Delta$.
\end{thm}

\subsection{Reducing to the irregular case}
It is easy to see that if there are irregular counterexamples to the
Borodin-Kostochka conjecture, then there are regular examples as well: take an
irregular counterexample $G$ clone it, add an edge between any vertex with degree
less than $\Delta(G)$ and its clone; repeat until you have a regular graph (from
\cite{molloy2002graph}).

But what about the converse?  If there are regular examples, must there be
(connected) irregular examples?  We'll see that the answer is yes, but we need
to decrease the maximum degree by one.

\begin{thm}\label{IrregularReduction}
Every graph satisfying $\chi \geq \Delta = k \geq 9$ either
contains $K_k$ or contains an irregular critical subgraph satisfying $\chi
= \Delta = k - 1$.
\end{thm}
\begin{proof}
Suppose not and choose a counterexample $G$ minimizing $\card{G}$. Then $G$ is
vertex critical. If every vertex in $G$ were contained in a $(k-1)$-clique, 
then Corollary \ref{TwoThirdsCliqueCor} would give a $K_k$ in $G$, impossible. 
Hence we may pick $v \in V(G)$ not in a $(k-1)$-clique. If $v$ is high, choose a
$(k-1)$-coloring $\pi$ of $G-v$ so that the color class $T$ of $\pi$
where $v$ has two neighbors is as large as possible; if $v$ is low, let $\pi$ be a
$(k-1)$-coloring of $G-v$ where some color class $T$ of $\pi$ is as large as
possible.  By symmetry, we may assume that $\pi(T) = k-1$.  

Now we have a $(k-1)$-coloring $\zeta$ of $H \DefinedAs G-T$ given by $\zeta(x)
= \pi(x)$ for $x \neq v$ and $\zeta(v) = k-1$.  Since $\chi(H) = k - 1$, the
maximality condition on $T$ together with Brooks' theorem gives $\Delta(H) =
k - 1$.  Note that $d_H(v) = k - 2$.  Let $H'$ be a $(k-1)$-critical subgraph of
$H$.  Then $H'$ must contain $v$ and hence is not $K_{k-1}$.  Since $d_{H'}(v) =
k-2$ and $\Delta(H') = k - 1$ (by Brooks' theorem), $H'$ is an irregular
critical subgraph of $G$ satisfying $\chi = \Delta = k - 1$, a contradiction.
\end{proof}

\begin{figure}[htb]
\centering
\begin{tikzpicture}[scale = 10]
\tikzstyle{VertexStyle}=[shape = circle,	
								 minimum size = 1pt,
								 inner sep = 3pt,
                         draw]
\Vertex[x = 0.257401078939438, y = 0.729450404644012, L = \tiny {}]{v0}
\Vertex[x = 0.232565611600876, y = 0.681758105754852, L = \tiny {}]{v1}
\Vertex[x = 0.282104313373566, y = 0.681911885738373, L = \tiny {}]{v2}
\Vertex[x = 0.383801102638245, y = 0.820650428533554, L = \tiny {}]{v3}
\Vertex[x = 0.358965694904327, y = 0.772958129644394, L = \tiny {}]{v4}
\Vertex[x = 0.40850430727005, y = 0.773111909627914, L = \tiny {}]{v5}
\Vertex[x = 0.506290018558502, y = 0.730872631072998, L = \tiny {}]{v6}
\Vertex[x = 0.48145455121994, y = 0.683180332183838, L = \tiny {}]{v7}
\Vertex[x = 0.530993163585663, y = 0.683334112167358, L = \tiny {}]{v8}
\Vertex[x = 0.440956592559814, y = 0.589494824409485, L = \tiny {}]{v9}
\Vertex[x = 0.416121125221252, y = 0.541802525520325, L = \tiny {}]{v10}
\Vertex[x = 0.46565979719162, y = 0.541956305503845, L = \tiny {}]{v11}
\Vertex[x = 0.317756593227386, y = 0.592694818973541, L = \tiny {}]{v12}
\Vertex[x = 0.292921125888824, y = 0.545002520084381, L = \tiny {}]{v13}
\Vertex[x = 0.342459797859192, y = 0.545156300067902, L = \tiny {}]{v14}
\Edge[](v2)(v1)
\Edge[](v2)(v0)
\Edge[](v1)(v0)
\Edge[](v4)(v3)
\Edge[](v5)(v3)
\Edge[](v5)(v4)
\Edge[](v7)(v6)
\Edge[](v8)(v6)
\Edge[](v8)(v7)
\Edge[](v10)(v9)
\Edge[](v11)(v9)
\Edge[](v11)(v10)
\Edge[](v13)(v12)
\Edge[](v14)(v12)
\Edge[](v14)(v13)
\Edge[](v3)(v0)
\Edge[](v4)(v0)
\Edge[](v5)(v0)
\Edge[](v3)(v2)
\Edge[](v4)(v2)
\Edge[](v5)(v2)
\Edge[](v3)(v1)
\Edge[](v4)(v1)
\Edge[](v5)(v1)
\Edge[](v3)(v6)
\Edge[](v4)(v6)
\Edge[](v5)(v6)
\Edge[](v3)(v7)
\Edge[](v4)(v7)
\Edge[](v5)(v7)
\Edge[](v3)(v8)
\Edge[](v4)(v8)
\Edge[](v5)(v8)
\Edge[](v6)(v9)
\Edge[](v7)(v9)
\Edge[](v8)(v9)
\Edge[](v6)(v11)
\Edge[](v7)(v11)
\Edge[](v8)(v11)
\Edge[](v6)(v10)
\Edge[](v7)(v10)
\Edge[](v8)(v10)
\Edge[](v12)(v10)
\Edge[](v13)(v10)
\Edge[](v14)(v10)
\Edge[](v12)(v9)
\Edge[](v13)(v9)
\Edge[](v14)(v9)
\Edge[](v12)(v11)
\Edge[](v13)(v11)
\Edge[](v14)(v11)
\Edge[](v12)(v2)
\Edge[](v13)(v2)
\Edge[](v14)(v2)
\Edge[](v12)(v1)
\Edge[](v13)(v1)
\Edge[](v14)(v1)
\Edge[](v12)(v0)
\Edge[](v13)(v0)
\Edge[](v14)(v0)
\end{tikzpicture}
\caption{$M_8$: A $C_5$ with vertices blown-up to triangles.}
\label{fig:M_8}
\end{figure}
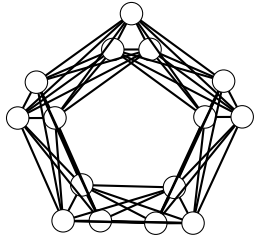

Since the only known critical (or connected even) counterexample to
Borodin-Kostochka for $\Delta = 8$ is regular (see Figure \ref{fig:M_8}) we
might hope that the following strengthened conjecture is true.

\begin{conjecture}\label{EightRegular}
Every critical graph with $\chi \geq \Delta = 8$ is regular.
\end{conjecture}

\subsection{Dense neighborhoods}
Here we show that the Borodin-Kostochka conjecture holds for graphs where each neighboorhood has ``most'' of its possible edges.  First, we need to convert high average degree in a neighborhood into a large clique in the neighborhood. We need the following extension of a fundamental result of Mader \cite{mader} (see Diestel \cite{diestel2010} for some history of this result).  We will also need $d_1$-choosability results from $\cite{mules}$ as well as some ideas for dealing with average degree in neighborhoods used in \cite{cranstonrabernclaw}.

\begin{lem}\label{MaderLemma}
For $k \geq 1$, every graph $G$ with $d(G) \geq 4k$ has a $(k+1)$-connected induced subgraph $H$ such that $d(H) > d(G) - 2k$.
\end{lem}

\begin{lem}\label{LowVertexHighAverageDegree}
If $B$ is a graph with $d(B) \geq \omega(B) + 2$, then $B$ has an induced subgraph $H$ such that $\join{K_1}{H}$ is $f$-choosable where $f(v) \geq d(v)$ for the $v$ in the $K_1$ and $f(x) \geq d(x) - 1$ for $x \in V(H)$.
\end{lem}
\begin{proof}
Let $B$ be such a graph.  Applying Lemma \ref{MaderLemma} with $k \DefinedAs 1$, we get a $2$-connected subgraph $H$ of $B$ with $d(H) > d(B) - 2 \geq \omega(B)$.  Since $H$ is $2$-connected, if it is not $d_0$-choosable, then it is either an odd cycle or complete.  The former is impossible since $d(H) \geq 3$, hence $H$ would be complete and we'd have the contradiction $\omega(H) > \omega(B)$.  Hence $H$ is $d_0$-choosable.  

Suppose $\join{K_1}{H}$ isn't $f$-choosable and let $L$ be a minimal bad $f$-assignment on $\join{K_1}{B}$.  
By Lemma \ref{LowSinglePair}, no nonadjacent pair in $H$ have intersecting lists and hence we must have $\sum_{v \in V(H)} \card{L(v)} \leq \card{Pot(L)}\omega(H)$.  Since for each $v \in V(H)$ we have $\card{L(v)} \geq d_H(v)$ and by the Small Pot Lemma we have $\card{Pot(L)} \leq \card{H}$, we must have $d(H) \leq \omega(H) \leq \omega(B) < d(H)$, a contradiction.
\end{proof}

\begin{lem}\label{VertexHighAverageDegree}
If $B$ is a graph with $d(B) \geq \omega(B) + 3$, then $B$ has an induced subgraph $H$ such that $\join{K_1}{H}$ is $d_1$-choosable.
\end{lem}
\begin{proof}
Let $B$ be such a graph.  Applying Lemma \ref{MaderLemma} with $k \DefinedAs 1$, we get a $2$-connected subgraph $H$ of $B$ with $d(H) > d(B) - 2 \geq \omega(B) + 1$.  As in the proof of Lemma \ref{LowVertexHighAverageDegree}, we see that $H$ is $d_0$-choosable. Suppose $\join{K_1}{H}$ is not $d_1$-choosable and let $L$ be a minimal bad $d_1$-assignment on $\join{K_1}{H}$.  Combining Lemma \ref{NeighborhoodPotShrink} with the same argument as in the proof of Lemma \ref{LowVertexHighAverageDegree} shows that $\card{Pot(L)} \leq \card{H}-1$.

Now, for $c \in Pot(L)$, we consider how big the color graphs $H_c$ can be.  All of the information comes from Lemma \ref{IntersectionsInB}. We have $\alpha(G_c) \leq 2$ for all $c \in Pot(L)$. First, suppose we have $c \in Pot(L)$ such that $\card{H_c} \geq \omega(H) + 3$.  Then, using Lemma \ref{IntersectionsInB}, we see that $\card{H_{c'}} \leq \omega(H)$ for all $c' \in Pot(L) - c$ and hence $\sum_{\gamma \in Pot(L)} \card{H_\gamma} \leq \card{H} + \parens{\card{Pot(L)} - 1}\omega(H) \leq \card{H}\omega(H) + \card{H} - 2\omega(H)$.  Now suppose we have $c \in Pot(L)$ such that $\card{H_c} = \omega(H) + 2$.  Then, using Lemma \ref{IntersectionsInB} again, we see that $\card{H_{c'}} \leq \omega(H) + 1$ for all $c' \in Pot(L) - c$ and hence $\sum_{\gamma \in Pot(L)} \card{H_\gamma} \leq 1 + \card{Pot(L)}(\omega(H) + 1) \leq \card{H}\omega(H) + \card{H} - \omega(H)$.

Therefore we must have $2\size{H} \leq \card{H}(\omega(H) + 1) - \omega(H)$ and hence $d(H) \leq \omega(H) + 1 < d(H)$, a contradiction.
\end{proof}

\begin{thm}\label{BKdense}
Every graph $G$ with $\omega(G) < \Delta(G)$ such that $d(G_v) \geq \frac23\Delta(G) + 4$ for each $v \in V(G)$ is $(\Delta(G)-1)$-colorable.
\end{thm}
\begin{proof}
Suppose note and let $G$ be a counterexample. Put $\Delta \DefinedAs \Delta(G)$.  Let $H$ be a $\Delta$-vertex-critical induced subgraph of $G$.  Then $\delta(H) \geq \Delta - 1$ and $H$ has no $d_1$-choosable induced subgraphs. By Theorem \ref{TwoThirdsCliqueCor}, we must have $v \in V(H)$ with $\omega(v) < \frac23\Delta + 2$. Suppose $d(H_v) < d(G_v)$. Then $d_H(v) = \Delta - 1$ and $\size{H_v} \geq \size{G_v} - (\Delta - 1)$; therefore, $d(H_v) > d(G_v) - 1 \geq \frac23\Delta + 3$.  Applying Lemma \ref{LowVertexHighAverageDegree} gives $\omega(v) > d(H_v) - 1 \geq \frac23\Delta + 2$, a contradiction.

Hence we must have $d(H_v) = d(G_v) \geq \frac23\Delta + 4$.  Applying Lemma \ref{VertexHighAverageDegree} gives $\omega(v) > d(H_v) - 2 \geq \frac23\Delta + 2$, a contradiction.
\end{proof}

\subsection{Bounding the order and independence number}
\begin{lem}\label{Onesies}
Let $G$ be a vertex critical graph with $\chi(G) = \Delta(G) + 1 - k$.  For every $v \in V(G)$ there is $H_v \unlhd G_v$ with:
\begin{enumerate}
\item $\card{H_v} \geq \Delta(G) - 2k$; and
\item $\delta(H_v) \geq \card{H_v} - (k+1)(\alpha(G) - 1) - 1$; and
\item $\size{H_v} \geq \card{H_v}\parens{\card{H_v}- (k+2)} - (k+1)\parens{\card{G} + 2k - (\Delta(G) + 1)}$.
\end{enumerate}
\end{lem}
\begin{proof}
Put $\Delta \DefinedAs \Delta(G)$. Pick $v \in V(G)$ and let $\pi$ be a $(\Delta - k)$-coloring of $G-v$.  Let $H_v$ be the subgraph of $G_v$ induced on $\setb{x}{N(v)}{\pi(x) \not \in \pi(N(v) - x)}$.  Plainly, $\card{H_v} \geq \Delta - 2k$.

By the usual Kempe chain argument, any $x, y \in V(H_v)$ must be in the same component of $C_{x,y} \DefinedAs G[\pi^{-1}(\pi(x)) \cup \pi^{-1}(\pi(y))]$.  Thus if $xy \not \in E(G)$, there must be a path of length at least $3$ in $C_{x,y}$ from $x$ to $y$ and hence some vertex of color $\pi(x)$ other than $x$ must have at least two neighbors of color $\pi(y)$ and some vertex of color $\pi(y)$ other than $y$ must have at least two neigbhors of color $\pi(x)$.  We say that such an intermediate vertex \emph{proxies} for $xy$.  Each $xy$ with $y \in V(H_v)$ must have some proxy $z_{xy} \in \pi^{-1}(\pi(x)) - x$ such that $z_{xy}$ proxies for at most $k+1$ total $xw$ with $w \in V(H_v)$, for otherwise we could recolor all of $xy$'s proxies, swap $\pi(x)$ and $\pi(y)$ in $x$'s component of $C_{x,y}$ and then color $v$ with $\pi(x)$ to get a $(\Delta-k)$-coloring of $G$.  We conclude that $x$ has at most $(k+1)(\card{\pi^{-1}(\pi(x))} - 1)$ non-neighbors in $H_v$.  This gives (2) immediately.

For (3), note that $\card{\pi(i)} \geq 2$ for each $i \in \irange{\Delta-k} - \pi(V(H_v))$ and hence $\sum_{j \in \pi(V(H_v))} \card{\pi^{-1}(j)} \leq \card{G} - 1 - 2(\Delta - k - \card{H_v})$.  Since $\size{H_v} \geq \sum_{j  \in \pi(V(H_v))} \parens{\card{H_v} - 1 - (k+1)(\card{\pi^{-1}(j)} - 1)}$, (3) follows.
\end{proof}

\begin{thm}\label{AlphaBound}
Every graph satisfies $\chi \leq \max\set{\omega, \Delta - 1, 4\alpha}$.
\end{thm}
\begin{proof}
Suppose not and choose a counterexample $G$ minimizing $\card{G}$.  Since none of the terms on the right side increase when we remove a vertex, $G$ is vertex critical.  Since the Borodin-Kostochka conjecture holds for graphs with $\alpha = 2$ and $\Delta \geq 9$, we must have $\alpha(G) \geq 3$ and hence $\Delta(G) \geq 13$.  By Lemma \ref{TwoThirdsCliqueCor}, there must be $v \in V(G)$ with $\omega(v) < \frac23 \Delta(G) + 2$.  Applying (2) of Lemma \ref{Onesies}, we get $H_v \unlhd G_v$ with $\card{H_v} \geq \Delta(G) - 2$ and $\delta(H_v) \geq \card{H_v} - 2\alpha(G) + 1$.  Since $\Delta(G) \geq \chi(G) \geq 4\alpha(G) + 1$, we have $\delta(H_v) \geq \card{H_v} - \frac{\Delta(G) - 1}{2} + 1 \geq \frac{\card{H_v} + 1}{2}$.  Applying Lemma \ref{neighborhood} shows that either $H_v = \join{K_3}{E_4}$ or $\omega(H_v) \geq \card{H_v} - 1$.  The former is impossible since $\Delta(G) > 9$. Therefore $\omega(v) \geq \omega(H_v) + 1 \geq \Delta(G) - 2 \geq \frac23 \Delta(G) + 2$ since $\Delta(G) \geq 12$, a contradiction.
\end{proof}

\begin{thm}\label{OrderBound}
Every graph satisfies $\chi \leq \max\set{\omega, \Delta - 1, \ceil{\frac{15 + \sqrt{48n + 73}}{4}}}$.
\end{thm}
\begin{proof}
Suppose not and choose a counterexample $G$ minimizing $\card{G}$.  Put $\Delta \DefinedAs \Delta(G)$ and $n \DefinedAs \card{G}$. Since none of the terms on the right side increase when we remove a vertex, $G$ is vertex critical. By Lemma \ref{TwoThirdsCliqueCor}, there must be $v \in V(G)$ with $\omega(v) < \frac23 \Delta + 2$.  Applying (3) of Lemma \ref{Onesies}, we get $H_v \unlhd G_v$ with with $\card{H_v} \geq \Delta - 2$ and $\size{H_v} \geq \card{H_v}\parens{\card{H_v}- 3} - 2\parens{n + 1 - \Delta}$.  By Lemma \ref{VertexHighAverageDegree}, we must have $d(H_v) < \frac23\Delta + 4$ and hence we have

\begin{align*}
\frac23\Delta + 4 &> 2\parens{\card{H_v}- 3} - \frac{4\parens{n + 1 - \Delta}}{\card{H_v}}\\
&\geq 2\parens{\Delta - 5} - \frac{4\parens{n + 1 - \Delta}}{\Delta - 2}.\\
\end{align*}

Simplifying a bit, we get $6(n-1) > (2\Delta - 15)(\Delta - 2)$.  Since $\Delta \geq \chi(G) \geq \frac{19 + \sqrt{48n + 73}}{4}$, we have $6(n-1) > (\frac{-11 + \sqrt{48n + 73}}{2})(\frac{11 + \sqrt{48n + 73}}{4}) = \frac{48n - 48}{8} = 6(n-1)$, a contradiction.
\end{proof}

\section{Coloring graphs when every vertex is in a big clique}

\subsection{The decomposition}
We need a partitioning result similar to Lemma \ref{partition} in the general case.  We deal with a set of pairwise intersecting $O_z$ by only using vertices in their intersection.  Since we need this intersection to be big in order to apply the independent transversal lemma, we need to limit the number of $O_z$ that can pairwise intersect.  For $k \geq 0$, let $\D_k$ be the collection of graphs without induced $d_k$-choosable
subgraphs. Again, for a graph $G$ and $t \in \IN$, we let $\CC_t$ be the maximal cliques in $G$ having at least $t$ vertices. 

We put off as much computation as possible until later, to this end, define $\fancy{U}(k, \omega, \Delta) \DefinedAs \max\set{\frac23 (\Delta + 1), \frac12 (\Delta + 3k + 2), \frac{2k}{2k+1}(\omega + k) - 1, \frac{k+1}{k+2}\omega + 2k + 1}$.

\begin{lem}\label{GeneralPartitionUniversalPrime}
Let $k \geq 1$ and suppose $G \in \D_k$. If $t \geq \fancy{U}(k, \omega(G), \Delta(G))$, then $\bigcup \CC_t$ can be partitioned into sets $D_1, \ldots, D_r$ so that for each $i \in \irange{r}$ each of the following holds:
\begin{enumerate}
\item $\card{D_i} \leq \omega(G[D_i]) + 2k$; and
\item $G[D_i]$ has at least $3k+1$ universal vertices; and
\item If $L$ is a maximum clique in $G[D_i]$, then for independent $I \subseteq D_i$ we have $\card{L \cap \bigcap_{v \in I} N(v)} \geq \card{L} - \card{I}(\card{L} + k - t)$; and
\item $G[D_i]$ has independence number at most $k+1$.
\end{enumerate}
\end{lem}
\begin{proof}
Put $\Delta \DefinedAs \Delta(G)$ and $\omega \DefinedAs \omega(G)$. If $\omega < t$, then $\bigcup \CC_t$ is empty and the lemma holds vacuously.  Hence we may assume $\omega \geq t$. Let $X_t$ be the intersection graph of $\CC_t$.  Since $t \geq \frac23 (\Delta + 1)$, $X_t$ is a disjoint union of complete graphs.  Let $\fancy{F}_1, \ldots, \fancy{F}_r$ be the components of $X_t$ and put $D_i \DefinedAs \cup \fancy{F}_i$ for $i \in \irange{r}$. 

Fix $i \in \irange{r}$.  Choose $L \in \fancy{F}_i$ with $\card{L} = \omega(G[D_i])$.  Put $A \DefinedAs D_i - L$.  

\textbf{Claim 1.} \textit{$\card{A} \leq 2k$ and $\card{\bigcap \fancy{F}_i} \geq 3k+1$.} Choose $S \subseteq A$ with $\card{S} = \min\set{2k+1, \card{A}}$.  Choose $\fancy{Q} \subseteq \fancy{F}_i$ such that $S \subseteq \bigcup \fancy{Q}$ so as to minimize $\card{\fancy{Q}}$.  Then for any $Q \in \fancy{Q}$ there exists $v_Q \in S \cap \parens{Q - \bigcup \parens{\fancy{Q} - \set{Q}}}$.  In particular, $\card{\fancy{Q}} \leq \card{S} = 2k+1$.  Suppose $\card{L \cap \bigcap \fancy{Q}} \geq 3k+1$.  Then by Lemma \ref{GeneralCliqueJoin} there is a clique with at least $\card{S} + \card{L} - 2k$ vertices intersecting $L$.  Since $L$ is a maximum size clique in $D_i$ we must have $\card{A} \leq \card{S} \leq 2k$.  But then $A = S$ and $\bigcap \fancy{F}_i = L \cap \bigcap \fancy{Q}$ since $A \subseteq \bigcup \fancy{Q}$.

Therefore to prove the claim it suffices to show that $\card{L \cap \bigcap \fancy{Q}} \geq 3k+1$.  If $\fancy{Q} = \emptyset$, then we are done since $\card{L} \geq 3k+1$. Hence $\card{\fancy{Q}} \geq 1$. For any $Q \in \fancy{Q}$, we have $\card{Q \cap L} \geq 3k+1$ since $t \geq \frac12 (\Delta + 3k + 2)$.  Hence we may apply Lemma \ref{TwoCliques} to get $\card{Q-L} + \card{L} - k \leq \card{L}$ and hence $\card{Q-L} \leq k$.  So we have $\card{Q\cap L} \geq \card{Q} - k$ for all $Q \in \fancy{Q}$.  If $\sum_{Q \in \fancy{Q}} \card{Q\cap L} \geq (\card{\fancy{Q}} - 1)\card{L} + 3k+1$, then applying Lemma \ref{BasicFiniteSets} gives the desired conclusion $\card{L \cap \bigcap \fancy{Q}} \geq 3k+1$.  Hence, if the claim fails, we must have
\begin{align*}
(\card{\fancy{Q}} - 1)\omega + 3k+1 &> \sum_{Q \in \fancy{Q}} \card{Q\cap L} \\
&\geq \sum_{Q \in \fancy{Q}} \parens{\card{Q} - k} \\
&\geq \card{\fancy{Q}}(t-k).
\end{align*}

Hence $t < \omega + k - \frac{\omega + 3k + 1}{\card{\fancy{Q}}} \leq \frac{2k}{2k+1}(\omega + k) - 1$ since $\card{\fancy{Q}} \leq 2k+1$, a contradiction.  This proves (1) and (2).

\textbf{Claim 2.} \textit{for independent $I \subseteq D_i$ we have $\card{L \cap \bigcap_{v \in I} N(v)} \geq \omega(G[D_i]) - \card{I}(\omega(G[D_i]) + k - t)$.}  For each $v \in I$ pick $Q_v \in \fancy{F}_i$ containing $v$ and put $\fancy{Q} \DefinedAs \setbs{Q_v}{v \in I}$. Note that $L \cap \bigcap_{v \in I} N(v) = L \cap \bigcap \fancy{Q}$.  As in the proof of Claim 1, we see that $\card{Q\cap L} \geq \card{Q} - k$ for each $Q \in \fancy{Q}$.  Therefore, if the claim fails, Lemma \ref{BasicFiniteSets} shows that we must have $\card{\fancy{Q}}(t-k) < (\card{\fancy{Q}} - 1)\card{L} + \card{L} - \card{I}(\card{L} + k - t)$ and hence $t < k + \card{L} - \card{I}\frac{(\card{L} + k - t)}{\card{I}} = t$, a contradiction.

\textbf{Claim 3.} \textit{$G[D_i]$ has independence number at most $k+1$.}  Suppose not and pick independent $I \subseteq D_i$ with $\card{I} = k+2$. By Claim 2, $\card{D_i \cap \bigcap_{v \in I} N(v)} \geq \card{L} - (k+2)(\card{L} + k - t)$.  Since $t \geq \frac{k+1}{k+2}\omega + 2k + 1$, we have $\card{D_i \cap \bigcap_{v \in I} N(v)} \geq (k+1)\omega - (k+1)\card{L} + (k+2)(2k+1) - k(k+2) \geq (k+1)(k+2)$.  Now Lemma \ref{BigIndependentJoin} gives a contradiction.
\end{proof}

\subsection{Doing the recoloring}

Again, to eliminate tiresome computation, put $\fancy{U}'(k, \omega, \Delta) \DefinedAs \max\set{\frac{k+2}{k+3}\Delta + 1, \fancy{U}(k, \omega, \Delta)}$.

\begin{lem}
For $k \geq 1$ and $\gamma \in \IN$, every graph $G$ with $\Delta(G) \leq \gamma$, $\omega(G) \leq \gamma - 2k$ and $\rho(G) \leq \gamma - k - \fancy{U}'(k, \omega(G), \gamma)$ is $(\gamma-k)$-colorable.
\end{lem}
\begin{proof}
Suppose the theorem fails and choose a counterexample first minimizing $\gamma$ and subject to that minimizing $\card{G} + \size{G}$. Put $\Delta \DefinedAs \Delta(G)$.  Plainly, we must have $\gamma > 2k$. Suppose $\Delta < \gamma$.  If $k' \geq 1$, then $G$ satisfies the hypotheses of the theorem with $\gamma' \DefinedAs \gamma - 1$ and $k' \DefinedAs k-1$.  By minimality of $\gamma$, $\chi(G) \leq \gamma'-k' = \gamma - k$, a contradiction.  If $k' = 0$, then Brooks' theorem gives a contradiction.  Hence $\Delta = \gamma$.  For any $v \in V(G)$ we have $\rho(G-v) \leq \rho(G)$ and hence the second minimality condition on $G$ implies that $G$ is $(\Delta + 1 - k)$-vertex-critical.

Therefore $\delta(G) \geq \Delta - k$ and $G \in \D_k$.  For any $v \in V(G)$, we have $\Delta - k - \omega(v) \leq d(v) - \omega(v) \leq \Delta - k - \fancy{U}'(k, \omega, \Delta)$ and hence $\omega(v) \geq  \fancy{U}'(k, \omega, \Delta)$. Applying Lemma \ref{GeneralPartitionUniversalPrime} with $t \DefinedAs \fancy{U}'(k, \omega, \Delta)$ we get a partition $D_1, \ldots, D_r$ of $\bigcup \CC_t = V(G)$.  For $i \in \irange{r}$, let $K_i$ be the universal vertices in $G[D_i]$, we know $\card{K_i} \geq 3k+1$.  Suppose every $x \in K_1$ has $N(x) \subseteq D_1$.  Put $j \DefinedAs \min_{x \in K_1} \Delta - d(x)$. Since $\card{K_1} \geq 3k+1$, applying Lemma \ref{GeneralCliqueJoinLow}, for every $x \in K_1$ we have $\Delta - 2j \geq \omega(G[D_i]) + 2(k-j) \geq \card{D_i} \geq d(x) + 1$, this contradicts the definition of $j$.  Thus we have $x \in K_1$ that has a neighbor $w \in V(G) - D_1$.

We claim that $xw$ is a critical edge in $G$.  Suppose otherwise that
$\chi(G - xw) = \Delta + 1 - k$.  Then by minimality of $G$ we must have $\rho(G-xw) >
\rho(G)$. Hence there is some $v \in N(x) \cap N(w)$ so that every
largest clique containing $v$ contains $xw$.  But $v$ is in some $D_j$ and all largest cliques containing $v$ are contained in $D_j$ and hence do not contain $xw$, a contradiction.  

Let $\pi$ be a $(\Delta-k)$-coloring of $G - xw$ chosen so that $\pi(x) = 1$
and so as to minimize $\card{\pi^{-1}(1)}$. Consider $\pi$ as a coloring of
$G-x$. One key property of $\pi$ we will use is that since $x$ got $1$ in the
coloring of $G - xw$ and $x \in K_1$, no vertex of $D_1 - x$ gets colored $1$ by $\pi$.

Now put $Z \DefinedAs \pi^{-1}(1)$ and for $z \in Z$, let $O_z$ be as defined
in Section \ref{recolorsection}.  By minimality of $\card{Z}$, each $z \in Z$ has at least one neighbor in every color class of $\pi$.  
Hence $z$ has two or more neighbors in at most $k+1$ of
$\pi$'s color classes. For each $z \in Z$ we have $i(z)$ such that $z \in D_{i(z)}$.  For each $a \in i(Z)$, let $L_a$ be a maximum clique in $G[D_a]$ and put $V_a \DefinedAs L_a \cap \bigcap_{z \in i^{-1}(a)} O_z$.  By Lemma \ref{GeneralPartitionUniversalPrime}, we have $\card{i^{-1}(a)} \leq k+1$ and

\begin{align*}
\card{V_a} &\geq \card{L_a} - \card{i^{-1}(a)}(\card{L_a} + k - t) - \card{i^{-1}(a)}(k + 1) \\
&= \card{L_a}(1-\card{i^{-1}(a)}) + \card{i^{-1}(a)}(t - 2k - 1) \\
&\geq \omega(1-\card{i^{-1}(a)}) + \card{i^{-1}(a)}(t - 2k - 1) \\
&= \omega + \card{i^{-1}(a)}(t - 2k - 1 - \omega) \\
&\geq \omega + (k+1)(t - 2k - 1 - \omega) \\
&\geq (k+1)(\frac{k+2}{k+3}\Delta - 2k) - k\omega \\
&\geq (k+1)(\frac{k+2}{k+3}\Delta - 2k) - k(\Delta - 2k) \\
&= \frac{(k+1)(k+2) - k(k+3)}{k+3}\Delta - 2k \\
&= 2\parens{\frac{\Delta}{k+3} - k}. \\
\end{align*}

Also $\card{D_a} \geq \card{L_a} + \card{i^{-1}(a)} - 1$. So, each $y \in V_a$ has at most $d(y) - (\card{D_a} - 1) \leq d(y) - \omega(y) - \card{i^{-1}(a)} + 2 \leq d(y) - \omega(y) + 1 \leq \Delta - k - \frac{k+2}{k+3}\Delta = \frac{\Delta}{k+3} - k$ neighbors outside $D_a$.  Now $\card{D_1} \geq t \geq \frac{k+2}{k+3}\Delta + 1$ and hence $x$ has at most $\frac{\Delta}{k+3}$ neighbors outside $D_1$.  Let $H$ be the subgraph of $G$ induced on the union of the $V_a$. Then, we may apply Lemma
\ref{SingletonSetTransversal} on $H$ to get an independent set $\set{x} \cup \set{v_z}_{z\in Z}$ where $v_z \in V_z$. Recoloring each $z \in Z$ with $\pi(z)$ and
coloring $x \cup \set{v_z}_{z \in Z}$ with $1$ gives a $(\Delta - k)$-coloring
of $G$, a contradiction.
\end{proof}

A little computation gives the following two more parsable results.

\begin{lem}\label{MainCorollary}
For $k \geq 1$ and $\gamma \in \IN$, every graph $G$ with $\Delta(G) \leq \gamma$, $\omega(G) \leq \gamma - 2k$ and $\rho(G) \leq \frac{\gamma}{2k+1} - (2k+1)$ is $(\gamma-k)$-colorable.
\end{lem}

\begin{thm}\label{MainSimpleCorollary}
Let $k \geq 1$. Every graph $G$ with $\omega(G) \leq \Delta(G) - 2k$ such that every vertex is in a clique on $\frac{2k}{2k+1}\Delta(G) + 2k + 1$ vertices is $(\Delta(G)-k)$-colorable.
\end{thm}

\subsection{Dense neighborhoods}
\begin{lem}\label{HighMinDegreeGivesClique}
Let $k \geq 1$. If $B$ is a graph with $\delta(B) \geq \frac{2k+1}{2k+2}\card{B} + k - 1$ such that $\join{K_1}{B}$ is not $d_k$-choosable, then $\omega(B) \geq \card{B} - 2k$.
\end{lem}
\begin{proof}
Let $L$ be a minimal bad $d_k$-assignment on $\join{K_1}{B}$. By the Small Pot Lemma, we have $\card{Pot(L)} \leq \card{B}$. Let $X \DefinedAs \set{\set{x_1, y_1}, \ldots, \set{x_s, y_s}}$ be a maximal set of pairwise disjoint independent sets of size $2$ in $B$.  Note that $s \leq \frac{\card{B}}{2}$.  Put $K \DefinedAs B - \cup X$.  Plainly, $K$ is complete.  Thus we are done if $s \leq k$.

Suppose $s \geq k + 1$.  Put $\epsilon \DefinedAs \frac{1}{2k + 2}$. For $i \in \irange{s}$ we have $\card{L(x_i)} + \card{L(y_i)} \geq d_B(x_i) + d_B(y_i) - 2k + 2 \geq 2(1-\epsilon)\card{B} \geq \card{B} + s$ and hence we may pick $s$ different colors $c_1, \ldots, c_s$ where $c_i \in L(x_i) \cap L(y_i)$.  Color both $x_i$ and $y_i$ with $c_i$ to get a list assignment $L'$ on $K$.

Each $v \in V(K)$ has at least $\delta(B) - (\card{K} - 1) \geq (1-\epsilon)\card{B} + k - \card{K} = 2s + k - \epsilon\card{B}$ neighbors in $\cup X$ and hence is joined to at least $s + k - \epsilon\card{B}$ pairs $\set{x_i, y_i}$.  Since $L$ is bad, we must have $s < \epsilon\card{B}$.

Now consider a pair $\set{x_i, y_i}$.  Vertices in $N(x_i) \cap N(y_i) \cap K$ have a color saved in $L'$ so we wish to show this set is big.  We have $\card{N(x_i) \cap N(y_i) \cap K} \geq 2(\delta(B) - (2s - 2)) - \card{K} = 2\delta(B) + 4 - \card{B} - 2s > (1-2\epsilon)\card{B} + 2k+2 - 2s$.

For $v \in V(K)$, let $l(v)$ count the number of $i \in \irange{s}$ such that $v$ is joined to $\set{x_i, y_i}$.  Then $\card{L'(v)} \geq \card{K} + l(v) - k$ for each $v \in V(K)$.  By Hall's theorem, we can complete the coloring if for all $0 \leq a \leq k$ we have $\card{\setb{v}{V(K)}{l(v) \geq k - a}} \geq a + 1$.  Thus it will suffice to show that for any $k-a$ indices $i_1, \ldots i_{k-a} \in \irange{s}$ we have $\card{\bigcap_{j \in \irange{k-a}} N(x_{i_j}) \cap N(y_{i_j}) \cap K} \geq a+1$.  If this were not the case for some $a$, then by Lemma \ref{BasicFiniteSets} we would have

\[(k-a)\parens{\frac{k}{k+1}\card{B} + 2(k+1) - 2s} < (k-a-1)(\card{B} - 2s) + a + 1.\]

A simple computation shows that this is impossible.  Hence we can complete the coloring, contradicting the fact that $L$ is bad.
\end{proof}

We can use this to turn Theorem \ref{MainSimpleCorollary} into a theorem about graphs with dense neighborhoods as follows.  Little effort was put into optimizing the bound $d(G_v)$ that we can get using Lemma \ref{HighMinDegreeGivesClique} since with a more developed $d_k$-choosability theory any such bound would be easily defeated (as in the $k=1$ case).

\begin{thm}\label{MainResult}
Let $k \geq 0$. Every graph $G$ with $\omega(G) \leq \Delta(G) - 2k$ such that $d(G_v) \geq \frac{6k^2}{6k^2 + 1}\Delta(G) + k + 6$ for each $v \in V(G)$ is $(\Delta(G)-k)$-colorable.
\end{thm}
\begin{proof}
Let $G$ be such a graph and put $\Delta \DefinedAs \Delta(G)$.  Suppose $G$ is not $(\Delta-k)$-colorable. If $k=0$, then Brooks' theorem gives a contradiction. If $k=1$, Theorem \ref{BKdense} gives a contradiction.

So, we must have $k \geq 2$.  Let $H$ be a $(\Delta + 1 - k)$-vertex-critical subgraph of $G$.  Then $\delta(H) \geq \Delta - k$ and hence $d(H_v) \geq d(G_v) - k$.  By Theorem \ref{MainSimpleCorollary}, there must be $v \in V(G)$ such that $\omega(v) < \frac{2k}{2k+1}\Delta + 2k + 1$.

Put $s \DefinedAs \frac{\Delta-k}{2k+1} - 6k$. Applying Lemma \ref{HighMinDegreeGivesClique} repeatedly on $\join{K_1}{H_v}$, gives a sequence $y_1, \ldots, y_s \in N_H(v)$ such that for each $i \in \irange{s}$ we have

\[\card{N_H(y_i) \cap (N_H(v) - \set{y_1, \ldots, y_{i-1}})} < \frac{2k+1}{2k+2}(\card{H_v} - i) + k - 1.\]

Hence the number of edges missing in $H_v$ is at least

\[-(k-1)s + \frac{1}{2(k+1)}\sum_{i \in \irange{s}} \parens{\Delta - k - i}.\]

On the other hand we have $d(H_v) \geq d(G_v) - k \geq \frac{6k^2}{6k^2 + 1}\Delta + 6$ and hence the number of edges missing in $H_v$ is at most $\binom{\Delta-k}{2} - \frac{3k^2}{6k^2 + 1}\Delta (\Delta - k) - 3(\Delta - k) < \frac{1}{2(6k^2 + 1)} (\Delta-k)^2 - \frac{7}{2}(\Delta - k)$.  Therefore, multiplying through by $2(k+1)(6k^2+1)$, we must have
\[-2(k^2-1)(6k^2+1)s + (6k^2+1)\parens{s(\Delta - k) - s(s+1)} \leq (k+1)(\Delta-k)^2 - 7(k+1)(6k^2+1)(\Delta - k).\]

Let's collect everything to the left side and look at the coefficients of the powers of $\Delta - k$ individually.  Plugging in for $s$, we have $\frac{6k^2 + 1}{2k+1} - \frac{6k^2+1}{(2k+1)^2} - (k+1) = \frac{8k^3 - 8k^2 - 3k - 1}{(2k+1)^2}$ for $(\Delta - k)^2$.  For $(\Delta-k)^1$ we get $\frac{-2(k^2-1)(6k^2+1)}{2k+1} - (6k-1)(6k^2+1) + 7(k+1)(6k^2+1) > 8(6k^2+1)$. Finally, for $(\Delta-k)^0$
we have $12k(k^2-1)(6k^2+1) - 36k^2 + 6k$.  Thus all of the coefficients are positive for $k \geq 2$, a contradiction.
\end{proof}

\begin{problem}
Develop $d_k$-choosability theory and improve the bound $d(G_v) \geq \frac{6k^2}{6k^2 + 1}\Delta(G) + k + 6$.  In particular, can the dependence on $k$ in $\frac{6k^2}{6k^2 + 1}$ be made linear?
\end{problem}

\section{List coloring lemmas}
Let $G$ be a graph.  A \emph{list assignment} to the vertices of $G$ is a
function from $V(G)$ to the finite subsets of $\mathbb{N}$.  A list assignment
$L$ to $G$ is \emph{good} if $G$ has a coloring $c$ where $c(v) \in L(v)$ for
each $v \in V(G)$.  It is \emph{bad} otherwise.  We call the collection of all
colors that appear in $L$, the \emph{pot} of $L$.  That is $Pot(L) \DefinedAs
\bigcup_{v \in V(G)} L(v)$.  For a subgraph $H$ of $G$ we write $Pot_H(L)
\DefinedAs \bigcup_{v \in V(H)} L(v)$. For $S \subseteq Pot(L)$, let $G_S$ be
the graph $G\left[\setb{v}{V(G)}{L(v) \cap S \neq \emptyset}\right]$.  We also
write $G_c$ for $G_{\{c\}}$. For $\func{f}{V(G)}{\IN}$, an $f$-assignment on $G$ is an
assignment $L$ of lists to the vertices of $G$ such that $\card{L(v)} = f(v)$
for each $v \in V(G)$.  We say that $G$ is \textit{$f$-choosable} if every
$f$-assignment on $G$ is good.  Given $\func{f}{V(G)}{\mathbb{N}}$, we have a partial order on the $f$-assignments to $G$ given by $L < L'$ iff $\card{Pot(L)} < \card{Pot(L')}$.  When we talk of \emph{minimal} $f$-assignments, we mean minimal with respect to this partial order.

We'll need a lemma about bad list assignments with minimum pot size proved in \cite{mules}.
Some form of this lemma which we call the \emph{Small Pot Lemma} has appeared independently in at least two places we know of---Kierstead \cite{kierstead2000choosability} and Reed and Sudakov \cite{ReedSudakov}.  We also use the following precursor to this lemma.

\begin{lem}\label{CannotColorSelfWithSelf}
Let $G$ be a graph and $\func{f}{V(G)}{\mathbb{N}}$.  Assume $G$ is not $f$-choosable and let $L$ be a minimal bad $f$-assignment. Assume $L(v) \neq Pot(L)$ for each $v \in V(G)$.  Then, for each nonempty $S \subseteq Pot(L)$, any coloring of $G_S$ from $L$ uses some color not in $S$.
\end{lem}

\begin{SmallPotLemma}
Let $G$ be a graph and $\func{f}{V(G)}{\mathbb{N}}$ with $f(v) < \card{G}$ for all $v \in V(G)$.  If $G$ is not $f$-choosable, then $G$ has a minimal bad $f$-assignment $L$ such that $\card{Pot(L)} < \card{G}$.
\end{SmallPotLemma}

We also need the notion of $d_k$-choosability from \cite{mules}.

\begin{defn}
Let $G$ be a graph and $r \in \mathbb{Z}$.  Then $G$ is \emph{$d_k$-choosable} if $G$ is $f$-choosable where $f(v) = d(v) - k$.
\end{defn}

An in-depth study of the $d_1$-choosable graphs was performed in \cite{mules}.  We only need a small portion of those results here and their generalization to $d_k$-choosability.  We use the following results on $d_1$-choosability.

\begin{lem}\label{K_tClassification}
For $t \geq 4$, $\join{K_t}{B}$ is not $d_1$-choosable iff $\omega(B) \geq \card{B} - 1$; or $t = 4$ and $B$ is $E_3$ or a claw; or $t = 5$ and $B$ is $E_3$.
\end{lem}

\begin{lem}\label{E2JoinWithSomeLow}\label{mixed}
Let $A$ be a graph with $\card{A} \geq 4$.  Let $L$ be a list assignment on $G \DefinedAs \join{E_2}{A}$ such that $\card{L(v)} \geq d(v) - 1$ for all $v \in V(G)$ and each component $D$ of $A$ has a vertex $v$ such that $\card{L(v)} \geq d(v)$.  Then $L$ is good on $G$.
\end{lem}

\begin{lem}\label{IntersectionsInB}
Let $A$ and $B$ be graphs such that $G \DefinedAs
\join{A}{B}$ is not $d_1$-choosable.  If either $\card{A} \geq 2$ or $B$ is
$d_0$-choosable and $L$ is a bad $d_1$-assignment on $G$, then
\begin{enumerate}
\item for any independent set $I \subseteq V(B)$ with $\card{I} = 3$, we have
$\bigcap_{v \in I} L(v) = \emptyset$; and
\item for disjoint nonadjacent pairs $\set{x_1, y_1}$ and $\set{x_2, y_2}$ at least one of the following holds
	\begin{enumerate}
	\item $L(x_1) \cap L(y_1) = \emptyset$;
	\item $L(x_2) \cap L(y_2) = \emptyset$;
	\item $\card{L(x_1) \cap L(y_1)} = 1$ and $L(x_1) \cap L(y_1) = L(x_2) \cap L(y_2)$.
	\end{enumerate}
\end{enumerate}
\end{lem}

\begin{lem}\label{NeighborhoodPotShrink}
Let $H$ be a $d_0$-choosable graph such that $G \DefinedAs \join{K_1}{H}$ is not
$d_1$-choosable and $L$ a minimal bad $d_1$-assignment on $G$.  If some
nonadjacent pair in $H$ have intersecting lists, then $\card{Pot(L)} \leq \card{H} - 1$.
\end{lem}

\begin{lem}\label{LowSinglePair}
Let $H$ be a $d_0$-choosable graph such that $G \DefinedAs \join{K_1}{H}$ is not
$f$-choosable where $f(v) \geq d(v)$ for the $v$ in the $K_1$ and $f(x) \geq
d(x) - 1$ for $x \in V(H)$. If $L$ is a minimal bad $f$-assignment on $G$, then
all nonadjacent pairs in $H$ have disjoint lists.
\end{lem}

\begin{lem}\label{neighborhood}
If $B$ is a graph with $\delta(B) \geq \frac{\card{B} + 1}{2}$ such that
$\join{K_1}{B}$ is not $d_1$-choosable, then $\omega(B) \geq \card{B} - 1$ or
$B = \join{E_3}{K_4}$.
\end{lem}
\begin{proof}
Suppose the lemma is false and let $L$ be a minimal bad $d_1$-assignment on $B$.
First note that if $B$ does not contain disjoint nonadjacent pairs $x_1, y_1$
and $x_2, y_2$, then $\omega(B) \geq \card{B} - 1$ or
$B = \join{E_3}{K_4}$ by Corollary \ref{K_tClassification}.

By Dirac's theorem, $B$ is hamiltonian and in particular $2$-connected. Since
$B$ cannot be an odd cycle or complete, $B$ is $d_0$-choosable.

By the Small Pot Lemma, $\card{Pot(L)} \leq \card{B}$.  Since $\card{L(x_1)} +
\card{L(x_2)} \geq \card{B} + 1$, their lists intersect and thus Lemma
\ref{NeighborhoodPotShrink} shows that $\card{Pot(L)} \leq \card{B} - 1$. But
then $\card{L(x_i) \cap L(y_i)} \geq 2$ for each $i$ and Lemma
\ref{IntersectionsInB} gives a contradiction.
\end{proof}

We'll need the following simple consequence of the pigeonhole principle.

\begin{lem}\label{BasicFiniteSets}
Let $r \in \IN_{\geq 1}$. If $S_1, \ldots, S_m$ are subsets of a finite set $T$ with $\card{S_i} \geq r$ for each $i \in \irange{r}$ and $\sum_{i \in \irange{m}} \card{S_i} \geq (m - 1)\card{T} + r$, then $\card{\bigcap_{i \in \irange{m}} S_i} \geq r$.
\end{lem}

\begin{lem}\label{BigIndependentJoin}
$\join{K_r}{E_{k+2}}$ is $d_k$-choosable when $k \geq 0$ and $r \geq (k+1)(k+2)$.
\end{lem}
\begin{proof}
Suppose not and let $L$ be a minimal bad $d_k$-assignment on $\join{K_r}{E_{k+2}}$ where $r = (k+1)(k+2)$.  By the Small Pot Lemma, we have $\card{Pot(L)} \leq r + k + 1$.  For each $v$ in the $K_r$ we have $\card{L(v)} = r + 1$ and for each $v$ in the $E_{k+2}$ we have $\card{L(v)} = r - k$.  Let $A$ be the colors appearing on the $K_r$ and $B$ the colors appearing on the $E_{k+2}$.  By Lemma \ref{CannotColorSelfWithSelf} we must have $B \subseteq A$.  Let $\set{X,Y}$ be a partition of $E_{k+2}$ with $\card{X} \geq \card{Y}$ and $\card{X} - \card{Y} \leq 1$.  Now $\sum_{x \in X} \card{L(x)} = \card{X}(r-k) \geq \card{X}(k^2+ 2k + 2) = (\card{X}-1)(k^2 + 4k + 3) - \card{X}(2k + 1) + k^2 + 4k + 3 \geq (\card{X} - 1)\card{Pot(L)} - \frac{k+3}{2}(2k + 1) + k^2 + 4k + 3 > (\card{X} - 1)\card{Pot(L)}$.  Hence Lemma \ref{BasicFiniteSets} gives $c_1 \in \bigcap_{x \in X} L(x)$.  Similarly, we have $c_2 \in \bigcap_{x \in Y} L(x)$.  Color all of $X$ with $c_1$ and all of $Y$ with $c_2$.  This leaves each vertex in $K_r$ with a list of size at least $r-1$.  We could complete the coloring if two of these lists were different.  Hence they are all the same and thus $L(x) = L(y)$ for all $x,y \in K_r$.  Since $B \subseteq A$, we must have $\card{Pot(L)} = r+1$.  But then $\sum_{x \in E_{k+2}} \card{L(x)} = (k+2)(r-k) \geq (k+2)(k^2 + 2k + 2) = (k+1)(k^2 + 3k + 3) + 1 > (k+1)\card{Pot(L)}$ and hence we have $c \in \bigcap_{x \in E_{k+2}} L(x)$.  Now we may color all of $E_{k+2}$ with $c$ and finish the coloring on $K_r$, a contradiction.
\end{proof}

The following lemma was proved in \cite{mules}.
\begin{lem}\label{GeneralCliqueJoin}
Let $k \geq 1$.  If $B$ is a graph such that $\join{K_{3k+1}}{B}$ is not $d_k$-choosable, then $\omega(B) \geq \card{B} - 2k$.
\end{lem}

With almost an identical proof we get the following extension that allows us to handle vertices of less than maximum degree more efficiently.
\begin{lem}\label{GeneralCliqueJoinLow}
Let $1 \leq j \leq k$.  Let $B$ be a graph and $L$ a list assignment on $G \DefinedAs \join{K_{3k+1}}{B}$ such that $\card{L(v)} \geq d(v)-j$ for $v \in V(K_{3k+1})$ and $\card{L(v)} \geq d(v)-k$ for $v \in V(B)$.  If $L$ is bad on $G$, then $\omega(B) \geq \card{B} - 2j$.
\end{lem}

To prove our final list coloring lemma, we need another tool from $\cite{mules}$.

\begin{lem}\label{ConnectedPot}
Fix $k \geq 1$. Let $A$ be a connected graph and $B$ an arbitrary graph such that $\join{A}{B}$ is not $d_k$-choosable.  Let $L$ be a minimal bad $d_k$-assignment on $\join{A}{B}$.  If $B$ is colorable
from $L$ using at most $\card{B} - k$ colors, then $\card{Pot(L)} \leq \card{A} + \card{B} - 2$.
\end{lem}

\begin{lem}\label{TwoCliques}
Let $k \geq 1$.  If $B$ is the complement of a bipartite graph and $\join{K_{3k+1}}{B}$ is not $d_k$-choosable, then $\omega(B) \geq \card{B} - k$.
\end{lem}
\begin{proof}
Suppose not, let $B$ be such a graph and $L$ a minimal bad $d_k$-assignment on $\join{K_{3k+1}}{B}$. Let $A$ be the $K_{3k+1}$. Let $M \DefinedAs \set{\set{x_1, y_1}, \ldots, \set{x_t, y_t}}$ be a maximum matching in the complement of $B$.  If $t \leq k$, then since $B$ is perfect we have $\omega(B) = \chi(B) \geq \card{B} - k$ giving a contradiction.

Hence $t \geq k+1$.  For $i \in \irange{t}$ we have $\card{L(x_i)} + \card{L(y_i)} \geq d_B(x_i) + d_B(y_i) + 4k+2 \geq \card{B} - 2 + 4k + 2 = \card{B} + 4k$ since $\alpha(B) \leq 2$.  By the Small Pot Lemma, we have $\card{Pot(L)} \leq \card{A} + \card{B} - 1 = \card{B} + 3k$.  Hence we have different colors $c_1, \ldots, c_k$ such that $c_i \in L(x_i) \cap L(y_i)$ for each $i \in \irange{k}$. For each such $i$, color each of $x_i, y_i$ with $c_i$.  Then we can complete the coloring on $B$ since $\card{A} \geq k+1$.  We just colored $B$ with at most $\card{B} - k$ colors and hence applying Lemma \ref{ConnectedPot} gives $\card{Pot(L)} \leq \card{B} + 3k - 1$.  So, now we can pick $c_{k+1} \in L(x_{k+1}) \cap L(y_{k+1}) - \set{c_1, \ldots, c_k}$, color each of $x_i, y_i$ with $c_i$ for $i \in \irange{k+1}$, then complete the coloring to $B$.  Now each vertex in $A$ has at least $k+1$ colors used twice on its neighborhood, so we can complete the coloring, a contradiction.
\end{proof}

\section{Independent transversals}
For completeness we include the proof of the main independent transveral lemma used above. In \cite{haxell2006odd}, Haxell and Szab{\'o} developed a technique for
dealing with independent transversals.  In \cite{haxell2011forming}, Haxell used
this technique to give simpler proof of her transversal lemma. The proof gives a bit more and we record that here.  This is just slightly more general than
the extension given in \cite{aharoni2007independent} by Aharoni, Berger and
Ziv. We write $\funcsurj{f}{A}{B}$ for a surjective function from $A$ to $B$.  Let $G$ be a graph.  For a $k$-coloring $\funcsurj{\pi}{V(G)}{\irange{k}}$ of $G$ and a subgraph $H$ of $G$ we say that $I \DefinedAs \set{x_1, \ldots, x_k} \subseteq V(H)$ is an $H$-independent transversal of $\pi$ if $I$ is an independent set in $H$ and $\pi(x_i) = i$ for all $i \in \irange{k}$.

\begin{lem}\label{BaseTransversalLemma}
Let $G$ be a graph and $\funcsurj{\pi}{V(G)}{\irange{k}}$ a proper $k$-coloring of
$G$.  Suppose that $\pi$ has no $G$-independent transversal, but for every $e
\in E(G)$, $\pi$ has a $(G-e)$-independent transversal. Then for every $xy \in
E(G)$ there is $J \subseteq \irange{k}$ with $\pi(x), \pi(y) \in J$ and an 
induced matching $M$ of $G\brackets{\pi^{-1}(J)}$ with $xy \in M$ such that:
\begin{enumerate}
  \item $\bigcup M$ totally dominates $G\brackets{\pi^{-1}(J)}$,
  \item the multigraph with vertex set $J$ and an edge between $a, b \in J$ for
  each $uv \in M$ with $\pi(u) = a$ and $\pi(v) = b$ is a (simple) tree.  In
  particular $\card{M} = \card{J} - 1$.
\end{enumerate}
\end{lem}
\begin{proof}
Suppose the lemma is false and choose a counterexample $G$ with
$\funcsurj{\pi}{V(G)}{\irange{k}}$ so as to minimize $k$.  Let $xy \in E(G)$.
By assumption $\pi$ has a $(G-xy)$-independent transversal $T$.  Note that we
must have $x,y \in T$ lest $T$ be a $G$-independent transversal of $\pi$.

By symmetry we may assume that $\pi(x) = k-1$ and $\pi(y) = k$. Put $X
\DefinedAs \pi^{-1}(k-1)$, $Y \DefinedAs \pi^{-1}(k)$ and $H \DefinedAs G -
N(\set{x, y}) - E(X,Y)$. Define $\func{\zeta}{V(H)}{\irange{k-1}}$ by $\zeta(v)
\DefinedAs \min\set{\pi(v), k-1}$. Note that since $x,y \in T$, we have
$\card{\zeta^{-1}(i)} \geq 1$ for each $i \in \irange{k-2}$.  Put $Z \DefinedAs
\zeta^{-1}(k-1)$. Then $Z \neq \emptyset$ for otherwise $M \DefinedAs \set{xy}$
totally dominates $G[X \cup Y]$ giving a contradiction.

Suppose $\zeta$ has an $H$-independent transversal $S$.  Then we have $z \in S
\cap Z$ and by symmetry we may assume $z \in X$.  But then $S \cup \set{y}$ is
a $G$-independent transversal of $\pi$, a contradiction.

Let $H' \subseteq H$ be a minimal spanning subgraph such that $\zeta$ has no
$H'$-independent transversal.  Now $d(z) \geq 1$ for each $z \in Z$ for
otherwise $T - \set{x,y} \cup \set{z}$ would be an $H'$-independent transversal
of $\zeta$.  Pick $zw \in E(H')$.  By minimality of $k$, we have $J \subseteq
\irange{k-1}$ with $\zeta(z), \zeta(w) \in J$ and an induced matching $M$ of
$H'\brackets{\zeta^{-1}(J)}$ with $zw \in M$ such that
\begin{enumerate}
  \item $\bigcup M$ totally dominates $H'\brackets{\zeta^{-1}(J)}$,
  \item the multigraph with vertex set $J$ and an edge between $a, b \in J$ for
  each $uv \in M$ with $\zeta(u) = a$ and $\zeta(v) = b$ is a (simple) tree.
\end{enumerate}

Put $M' \DefinedAs M \cup \set{xy}$ and $J' \DefinedAs J \cup \set{k}$.
Since $H'$ is a spanning subgraph of $H$, $\bigcup M$ totally dominates
$H\brackets{\zeta^{-1}(J)}$ and hence $\bigcup M'$ totally dominates
$G\brackets{\pi^{-1}(J')}$.  Moreover, the multigraph in (2) for $M'$ and $J'$
is formed by splitting the vertex $k-1 \in J$ in two vertices and adding an edge
between them and hence it is still a tree.  This final contradiction proves the
lemma.
\end{proof}

\section{Acknowledgements}
\noindent Thanks to Dan Cranston for finding an error in my original demonstration of Lemma \ref{MainBKLemma} which led to a simpler proof.

\bibliographystyle{amsplain}
\bibliography{GraphColoring}
\end{document}